\documentclass[11pt,reqno]{amsart}

\usepackage{graphicx}

\newcommand{\phip}{\phi_{P^o} }
\newcommand{\szego}{Szeg\"o }

\newcommand{\T}{{\mathbf T}^m}

\newcommand{\kahler}{K\"ahler }

\newcommand{\PP}{{\mathbb P}}

\newcommand{\QQ}{\mathbb{Q}}

\newcommand{\R}{{\mathbb R}}
\newcommand{\C}{{\mathbb C}}

\newcommand{\Z}{{\mathbb Z}}

\newcommand{\N}{{\mathbb N}}
\newcommand{\CP}{\C\PP}

\newcommand{\dbar}{\bar\partial}
\newcommand{\ddbar}{\partial\dbar}

\newcommand{\half}{{\frac{1}{2}}}

\newcommand{\dcal}{\mathcal{D}}

\newcommand{\fcal}{\mathcal{F}}

\newcommand{\hcal}{\mathcal{H}}

\newcommand{\mcal}{\mathcal{M}}

\newcommand{\pcal}{\mathcal{P}}

\newcommand{\ocal}{\mathcal{O}}

\def    \half   {{\frac{1}{2}}}

\def    \Z  {{\mathbb Z}}
\def    \R  {{\mathbb R}}
\def    \C  {{\mathbb C}}

 \def   \half   {{\frac{1}{2}}}

 \def    \Re     {{\operatorname{Re}}}

\newtheorem{maintheo}{{\sc Theorem}}

\newtheorem{theo}{{\sc Theorem}}[section]

\newtheorem{lem}[theo]{{\sc Lemma}}
\newtheorem{prop}[theo]{{\sc Proposition}}

\newenvironment{rem}{\medskip\noindent{\it Remark:\/} }{\medskip}
\newenvironment{defin-no-number}{\medskip\noindent{\it Definition:\/} }{\medskip}

\def\text{\textstyle}

\def\max{{\operatorname{max}}}

\title[
Off-diagonal decay of toric Bergman kernels
]
{Off-diagonal decay of toric Bergman kernels
}

\author{Steve Zelditch }

\address{Department of Mathematics, Northwestern  University,
Evanston, IL 60208, USA} \email{ zelditch@math.northwestern.edu}

\thanks{Research partially supported by NSF grant DMS--1541126 }

\begin{document}

\maketitle
\begin{abstract}  We study the off-diagonal decay of Bergman kernels
$\Pi_{h^k}(z,w)$  and Berezin kernels $P_{h^k}(z,w)$ 
for ample invariant line bundles over compact toric projective \kahler manifolds of dimension $m$.
When the metric is real analytic, $P_{h^k}(z,w) \simeq k^m \exp - k D(z,w)$ 
where $D(z,w)$ is the diastasis. When the metric is only $C^{\infty}$ this
asymptotic cannot hold for all $(z,w)$ since the diastasis is not even defined for all $(z,w)$ close to the diagonal. Our main
result is that for  general $C^{\infty}$ metrics, $P_{h^k}(z,w) \simeq k^m \exp - k D(z,w)$ as long as $w$ lies on the $\R_+^m$-orbit of $z$, and
for general $(z,w)$, $\limsup_{k \to \infty} \frac{1}{k} \log P_{h^k}(z,w)
 \leq  - D(z^*,w^*)$ where $D(z, w^*)$ is the diastasis between $z$ and the translate of $w$  by  $(S^1)^m$ to the $\R_+^m$ orbit of $z$. These results
 are complementary to Mike Christ's negative results showing that 
 $P_{h^k}(z,w) $ does not have off-diagonal decay at ``speed'' k if $(z,w)$
 lie on the same $(S^1)^m$-orbit.
\end{abstract}

The problem we are concerned with in this note  is
to find conditions on a positive Hermitian line bundle $(L, h) \to (M, \omega)$
over a K\"ahler manifold so that the Szeg\"o kernel $\Pi_{h^k}(z,w)$  for $H^0(M, L^k)$
has exponential decay at speed $k$. 
We denote the 
  {\it Berezin kernel} or  {\it normalized \szego kernel} by
\begin{equation}\label{PN} P_{h^k}(z,w):=
\frac{|\Pi_{h^k}(z,w)|}{\Pi_{h^k}(z,z)^\frac 12 \Pi_{h^k}(w,w)^\frac
12}\;.\end{equation}

 \bigskip
 
 \noindent{\bf Problem}  {\it Let $  D_h^*(z,w)$ be  the upper semi-continuous regularization of
 \begin{equation} \label{LIMSUP} \limsup_{k \to \infty} \frac{1}{k} (- \log P_{h_k}(z,w)). \end{equation}
 Determine $D_h^*(z,w)$ and in particular determine when it is non-zero. }
 \bigskip
 
 The minus sign is due to the fact that \eqref{PN} is pluri-superharmonic in $z$ and we prefer to deal with pluri-subharmonic functions.
 It is known that for real analytic metrics, $P_{h^k}(z,w) \leq C e^{- k D(z,w)}$ for points $(z,w)$ sufficiently close to the diagonal,
 where $D(z,w)$ is the so-called Calabi diastasis (\S\ref{CD}). Near the diagonal, $D(z,w) \sim |z - w|^2$.  For general smooth metrics,
 the sharpest general result is that $P_{h^k}(z,w)\leq C e^{- A\sqrt{k} \log k }$ for all $A < \infty$ \cite{Chr13, Chr13B}.
This raises the question of whether, for $C^{\infty}$ but not real analytic metrics, $D_h^*(z,w)$ can be strictly negative off
the diagonal.

A stronger condition which arises in several problems (see \cite{RZ})  is whether there exists a pointwise limit
\begin{equation}\label{PTWISE} \frac{1}{k} \log P_{h^k}(z,w) \to - D(z,w) \end{equation}
for some function $D(z,w)$ defined near the diagonal in $M \times \bar{M}$.
If the metric is real analytic, then such a limit does exist and  $D(z,w)$ is
the Calabi diastasis of the metric (see \S \ref{CD}).  The diastasis
is the real  part of the off-diagonal analytic continuation of a local Kaehler potential
of $\omega$ \cite{Ca53}.
Existence of a pointwise limit near the diagonal would be surprising if the metric is $C^{\infty}$ but not real analytic, since it would define a Calabi diastasis even
though the Kaehler potential admits no analytic continuation. One might therefore
expect the neighborhood of $z$ in which the limit \eqref{PTWISE} exists to be
the largest neighborhood of $z$ in which the Kaehler potential $\phi$ admits an analytic
 continuation.

In this note we study these questions in the case of a positive Hermitian holomorphic toric line bundle $(L, h) \to (M, \omega_h)$
with $C^{\infty}$ metric $h$. As recalled in \S \ref{TORSECT}, a toric \kahler manifold is a
\kahler  manifold  on which the complex torus
$(\C^*)^m$ acts holomorphically with an open orbit $M^o$. We denote by $\T$ the
underlying real torus and by $\R_+^m$ the real subgroup of $(\C^*)^m$. 
We denote a point by  $z = e^{\rho/2 +i\varphi} m_0$ where $e^{\rho/2}$ denotes the $\R_+^m$
action and $e^{i \varphi} $ denotes the $\T$ action. Let $h = e^{-\phi}$ in a toric holomorphic frame over $M^o$.
As recalled in \S \ref{TCD},  $\phi(e^{\rho/2}) = \tilde{\phi}(\rho)$
on the open orbit, where $\tilde{\phi}$ is convex.

Given two points $z = e^{\rho_1/2 + i \theta_1}, w = e^{\rho_2 + i \theta_2}$ we denote
by $z^* = e^{\rho_1/2},$ resp. $ w^* = e^{\rho_2/2}$ the unique  point on the $\R_+^m$ orbit of $m_0$
which lie on the same $\T$ orbit as $z$, resp. $w$. 
Our main result is that  $D^*_h(z,w) \leq - D(z^*, w^*) $ where $D(z,w)$ 
is the Calabi diastasis  (see \S \ref{CD} and \S \ref{TCD}).





\begin{maintheo} \label{<}Let $(L, h) \to (M, \omega)$ be a positive Hermitian toric line bundle over a toric
Kaehler manifold. Then if $z, w \in M^o$,
$$\limsup_{k \to \infty} \frac{1}{k} \log P_{h^k}(z,w)
 \leq  - D(z^*,w^*) \leq 0, $$
 with $D(z^*,w^*)  = 0$  if and only if $z^* = w^*$.
 Furthermore, if 
$z = e^{\rho_1/2 + i \theta}, w = e^{\rho_2/2 + i \theta}$ lie on the same $\R_+^m$ orbit, then one has the pointwise limit
\eqref{PTWISE}, 
$$\lim_{k \to \infty} \frac{1}{k} \log P_{h^k}(z,w) = - D(z,w)
  $$
 The same asymptotics and upper bounds are valid for all $z,w \in M$
 in the domain of $D(z,w).$

\end{maintheo}\label{MAIN}

Thus, $\limsup_{k \to \infty} \frac{1}{k} \log P_{h^k}(z,w)
 < 0$ except on the real codimenson  $m$ subset 
 $\mcal: = \{(z,w) \in M \times M: \exists \vec \theta: e^{i \vec \theta} w = z\}$ of $(z,w)$ which lie on the same $\T$ orbit, i.e.
 have the same $\rho$ coordinates. We say  that $P_{h^k}(z,w)$ has
 exponential decay at speed $k$ except on $\mcal$. In a closely related setting,
 M. Christ showed that one does not have exponential decay on the analogue
 of the set $\mcal$. More precisely, in  
Theorem 2.1 of \cite{Chr13} he proved that  if  there exists an open set $U \subset \C^m$ so that, for any $\delta > 0$, there exists a sequence $k_{\nu}
\to \infty$ such that $|B(z,z') \leq e^{- \epsilon k_{\nu}}$ for some $\epsilon > 0$
and for all  $(z,z') \in U$, $|z - z'| \geq \delta$, then the Kaehler potential $\phi$ is real analytic on $U$. The points where it does not have such decay belong to the
analogue of $\mcal$. The  result is reviewed in \S \ref{CHRISTSECT} and related
 to Theorem \ref{MAIN}.  Although Christ's result is not stated or proved for
 toric \kahler manifolds it seems likely that his proof can be modifed to apply
 to them.

Regarding pointwise  asymptotics \eqref{PTWISE} , the locus of pairs $(z,w)$ where this is proved  lie on the same $\R_+^m$ orbit and are also of real codimension $m$ in $M \times M$,
and so the set of pairs  where we have proved existence of  a pointwise limit is very sparse. 
 Although we do not prove it
here, it is doubtful that the pointwise limit exists when the angular coordinates of $z$ and $w$ are unequal,
i.e. if  $z, w \in M^o$ but do not lie on the same $\R_+^m$ orbit. In effect this would require an analytic continuation
of $\phi(z)$ to the  full off-diagonal. 

Even in simple  real analytic cases, the far off-diagonal behavior of the \szego kernel or Calabi diastasis can be singular. 
For instance, when  $M = \CP^m$ equipped with its Fubini-Study metric, the \szego kernel is given in homogeneous coordinates $Z \in \C^{m+1} \backslash \{0\}$
by 
$$\Pi_{h_{FS}^k}(Z,W)  =  \frac{ (Z \cdot \bar{W})^k }{|Z|^k |W|^k}. $$
The Calabi diastasis of the Fubini-Study metric is 
$$-D_{FS} =  \log  \frac{ |(Z \cdot \bar{W})| }{|Z| |W|} = \frac{1}{k} \log \Pi_{h_{FS}^k}(Z,W) . $$
When  $Z \cdot \bar{W} = 0$ the diastasis is infinite and the \szego kernel vanishes. For $\CP^1$
in affine coordinates, this occurs when $w = - \frac{1}{\bar{z}}$, i.e. $z = e^{\rho/2 + i \theta}, w = e^{-\rho/2 - i \theta}.$
Such points do not lie on the same $\R_+$ orbit and in higher dimensions anti-podal points  $(Z, W): Z \cdot \bar{W} = 0$
never lie on the same $\R_+^m $ orbit.

\subsection{Remarks on the proof}

The proof is based on the explicit formula for 
the Bergman kernels
of a positive toric Hermitian line bundle over a toric Kaehler manifold $M^m$
of dimension $m$. As is 
well known, there exists a moment map $\mu: M \to P$ for the torus action which
maps $M$ to a convex Delzant polytope $P \subset \R^m$.  The open orbit $M^o$ of the
complexifed $(\C^*)^m$-action maps to the interior of $P$, while the inverse image
$\mu^{-1}(\partial P)$ of the boundary of $P$ is the divisor at
infinity $\dcal \subset M$. The toric holomorphic sections
$H^0(M, L^k)$ correspond to monomials $z^{\alpha}$ with lattice points
$\alpha \in k P \cap \Z^m. $ In  \S \ref{TBK} and in Lemma \ref{BNTORIC} we review the proof
that  for $z,w$ in the open orbit $M^o$, the  toric Bergman kernels for a toric Hermitian metric $h$ are given by 
\begin{equation} \label{Bhkintro} B_{h^k}(z, w)=
\sum_{ \alpha \in k P}  \frac{z^{\alpha}
\bar{w}^{\alpha}}{Q_{h^k}(\alpha)}, \;.\end{equation} 
where $Q_{h^k}(\alpha)$ is the $L^2$-norm squared of $z^{\alpha}$ with respect
to the natural inner product $\rm{Hilb}_k(h)$ defined by $h$. As recalled in \S \ref{TBK},
$B_{h^k}$ is the local expression of $\Pi_{h^k}$ relative to a local holomorphic  frame $e^k$ of $L^k$ over $M^o$.

The essential point  is that the off-diagonal
value $\log P_{h^k}(e^{\rho_1/2}, e^{\rho_2/2})$ equals the diagonal value at $e^{\half(\rho_1 + \rho_2)}. $
 The point  $e^{\half(\rho_1 + \rho_2)}$ is a kind of midpoint with
respect to this action of the orbit from $z$ to $w$. The well known diagonal Bergman kernel asymptotics
thus determines special cases of the off-diagonal asymptotics. Simple estimates extend the result to $\T$-orbits
of the real slice.

\subsection{Further questions}

A natural question is whether $D^*(z,w) = D(z^*,w^*)$ for general $z,w$. Although the factor $e^{i k \langle \theta_1 - \alpha_2, \alpha\rangle}$
is fast oscillating and can cause significant cancellation, it is not clear whether it can change the exponential decay rate for every
sequence of $k$'s. Given $(z,w)$ it is known \cite{SoZ2} that the sum concentrates exponentially fast around one point $x(z,w) \in P$.
As long as $\langle \theta_1 - \alpha_2, x(z,w)\rangle$ is irrational, there exists  subsequence $k_n$ so that $e^{i k_n \langle \theta_1 - \alpha_2, x(z,w)\rangle} \to 1$.
This suggests that $D^*(z,w)= D(z^*, w^*)$. However the exponential concentration is not well enough localized to prohibit fast oscillation at nearby
lattice points to $x(z,w)$, and  so the suggestion is not necessarily plausible.


\subsection{Remarks on general metrics}

The original motivation for  the off-diagonal decay problem comes from its applications to   finite dimensional approximations to solutions of the initial value problem
for geodesics in the space of Kaehler metrics in \cite{RZ} and for the boundary
value problem in \cite{PhSt, SoZ}. In \cite{RZ} it is conjectured that a solution $\phi_{\tau}(z)$ of the
initial value problem exists up to time $T$ if and only if the limit of $\frac{1}{k} U_{h^k}(i \tau, z,z)$
tends to $\phi_{\tau}(z)$ as $k \to \infty$. Without going into the details, the kernel $U_{h^k}(i \tau, z,z)$ is the
value of the Bergman kernel on a certain off-diagonal set. Obviously, the geodesic cannot exist unless
the Bergman kernel decays at speed $k$ on this set and has a limit. If the limit does not
exist, one may take the limsup as in \eqref{LIMSUP}. Its upper semi-continuous regularization is
always well defined and gives a pluri-subharmonic sub-solution of the geodesic equation. Again, this
solution is trivial if the speed of decay is smaller than $k$. When it equals $k$ but does not have
a pointwise limit, then the regularized limsup will define a singular subsolution, and it is not clear
whether it is a weak solution or not.
  We refer to \cite{RZ} for background and the precise
conjecture. For different reasons,
the off-diagonal decay rate of Bergman kernels has  has been studied by 
M. Christ in several papers \cite{Chr03,Chr13,Chr13B}. 

it would be interesting to have necessary or sufficient conditions under
which  the Bergman kernels may satisfy \eqref{LIMSUP} on all but
a large codimension set 
when the metrics are only $C^{\infty}$. We tend to doubt that they exist except
in very special (perhaps symmetric) situations such as toric \kahler metrics.

\section{\label{BACKGROUND} Background on  Bergman kernels}

The \szego (or Bergman) kernels of a positive Hermitian
line bundle $(L, h) \to (M, \omega)$ over a \kahler manifold are
the kernels of the orthogonal projections $\Pi_{h^k}: L^2(M, L^k)
\to H^0(M, L^k)$ onto the spaces of holomorphic sections with
respect to the inner product $\rm{Hilb}_k(h)$ defined by
\begin{equation} \label{HILBDEF}  (s_1, s_2)_{\rm{Hilb}_k(h)} = \int_M (s_1(z), s_2(z))_{h^k}\;
\omega_h^m/m!. \end{equation}

  Thus, we
have
\begin{equation} \Pi_{h^k} s(z) = \int_M \Pi_{h^k}(z,w) \cdot s(w)
\frac{\omega_h^m}{m!}, \end{equation} where the $\cdot$ denotes
the $h^k$-hermitian inner product at $w$.
 Let $e_L$ be a local holomorphic  frame for $L \to M$ over an
open set $U \subset M$ of full measure,  and let $\{s^k_j=f_j
e_L^{\otimes k}:j=1,\dots,d_k\}$ be an orthonormal basis for
$H^0(M,L^k)$ with $d_k = \dim H^0(M, L^k)$.  Then the \szego
kernel can be written in the form
\begin{equation}\label{szego}  \Pi_{h^k}(z, w): = B_{h^k}
(z, w)\,e_L^{\otimes k}(z) \otimes\overline {e_L^{\otimes
k}(w)}\,,\end{equation} where
\begin{equation}\label{BN} B_{h^k}(z, w)=
\sum_{j=1}^{d_k}f_j(z) \overline{f_j(w)}\;.\end{equation}

The contraction 
\begin{equation} \label{PIB} \Pi_{h^k}(z,z) = B_{h^k}(z,z) ||e||_{h^k}^2 = B_{h^k} e^{- k \phi} \end{equation}
balances the exponential growth/decay of the two factors into a power 
expansion,
\begin{equation} \label{TYZ}  \Pi_{h^k}(z,z) = B_{h^k} (z,z) e^{- k \phi}(z) = a_0 k^m + a_1(z) k^{m-1} + a_2(z) k^{m-2} +
\dots \end{equation} where $a_0$ is constant;  see \cite{Z}.

\subsection{\label{CD} \kahler potential and Calabi diastasis}

Let $h = e^{-\phi}$. When $\phi$ is real analytic, 
the analytic continuation $\phi(z,w)$ of the \kahler potential was
used by Calabi \cite{Ca53} in the analytic case to define a
\kahler distance function, known as the
 `Calabi diastasis
function' \begin{equation} \label{DIASTASIS} D(z,w): =  \phi(z,w)
+ \phi(w,z) - (\phi(z) + \phi(w)).
\end{equation}
Calabi showed that
\begin{equation}\label{HESSDIA}   D(z,w) = d (z,w)^2 +
O(d(z,w)^4),\;\; dd^c_w D(z,w)|_{z = w} = \omega.
\end{equation}
 
When $\phi$ is only  $C^{\infty} $  one may try to replace $\phi(z,w)$ by the   {\it almost analytic extension}
$\phi(z,w)$  of $\phi$ to $M \times M$, defined near the totally real anti-diagonal
$(z, \bar{z}) \in M \times M$ by
\begin{equation} \label{AAE} \phi_{\C} (x + h, x + k) \sim
 \sum_{\alpha, \beta} \frac{\partial^{\alpha + \beta}
\phi}{\partial z^{\alpha}
\partial \bar{z}^{\beta}} (x) \frac{h^{\alpha}}{\alpha!}
\frac{k^{\beta}}{\beta!}. \end{equation}  It is  a smooth
function defined in a small neighborhood $(M \times M)_{\delta} =
\{(z,w): d (z,w)< \delta\} $ of the anti-diagonal in $M \times M$, with the right side of (\ref{AAE}) as its $C^{\infty}$
Taylor expansion along the anti-diagonal, for which  $\dbar
\phi(z,w) = 0$ to infinite order on the anti-diagonal. But the almost
analytic extension is not unique and the associated Calabi diastasis can only
have a geometric meaning as a germ on the diagonal.

\section{\label{TORSECT} Toric Bergman kernels}

We briefly review toric \kahler manifolds and their Bergman kernels; the exposition is similar to that of \cite{SoZ}.
 A toric \kahler manifold is a
\kahler  manifold $(M, J, \omega)$ on which the complex torus
$(\C^*)^m$ acts holomorphically with an open orbit $M^o$. Choosing
a basepoint $m_0$ on the  open orbit identifies $M^o \equiv
(\C^{*})^{m}$. We define holomorphic coordinates on the open orbit 
by giving  the point $z = e^{\rho/2 +i\varphi} m_0$
the  coordinates
\begin{equation} \label{OPENORBCOORDS}
z=e^{\rho/2 +i\varphi} \in (\C^{*})^{m},\quad \rho, \varphi \in
\R^{m}. \end{equation} The real torus $\T \subset (\C^*)^m$ acts
in a Hamiltonian fashion with respect to $\omega$. Its moment map
$\mu = \mu_{\omega}: M \to P \subset {\bf t}^* \simeq \R^m$ (where
${\bf t}$ is the Lie algebra of $\T$) with respect to $\omega$
defines a singular torus fibration over a convex lattice polytope
$P$.  

When $M$ is a smooth toric manifold, $P$ is a Delzant polytope defined by a set of linear inequalities
$$\ell_r(x): =\langle x, v_r\rangle-\lambda_r \geq 0, ~~~r=1, ..., d, $$
where $v_r$ is a primitive element of the lattice and
inward-pointing normal to the $r$-th $(m-1)$-dimensional facet
$F_r = \{\ell_r = 0\}$  of $P$. By  a facet we mean an $m-1$-
 dimensional face of $\partial P$For $x \in
\partial P$ we denote by
$$\fcal(x) = \{r: \ell_r(x) = 0\}$$
the set of facets containing $x$.

Over the open orbit one t has a torus fibration,
$$\mu: M^o \simeq P^o \times \T.$$
We let $x$ denote the Euclidean coordinates on $P$. 
The components
$(I_1, \dots, I_m)$  of the moment map are called action variables
for the torus action. The symplectically dual variables on $\T$
are called the angle variables. Given a basis of ${\bf t}$ or
equivalently of the action variables,  we denote by
$\{\frac{\partial}{\partial \theta_j}\}$ the corresponding
generators (Hamiltonian vector fields) of the $\T$ action. Under
the complex structure $J$, we also obtain generators
$\frac{\partial}{\partial \rho_j}$ of the $\R_+^m$ action.

The  generators of the $\T$
action vanish on the divisor at infinity,  $\dcal$. If $p \in \dcal$ and
$\T_p$ denotes the isotropy group of $p$, then the generating
vector fields of $\T_p$ become linearly dependent at $P$.


\subsection{\label{SOC}Monomial sections}

There exists an orthonormal basis $\chi_{\alpha}$ of $H^0(M, L^k)$ given by eigensections
of the $\T$ action for $\alpha \in k \overline{P}$. On the open orbit in the coordinates above\footnote{Throughout the article we use standard
multi-index notation, and put $|\alpha| = \alpha_1 + \cdots +
\alpha_m$.},
$$\chi_{\alpha}(z) =  z^{\alpha} = z_1^{\alpha_1} \cdots
z_n^{\alpha_n}.$$

Let $\# P $ denote the number of lattice points $\alpha \in \N^m
\cap P$. We denote by $L \to M$ the invariant line bundle obtained
by pulling back $\ocal(1) \to \CP^{\# P - 1}$ under the monomial
embedding defining $M$.
 A natural basis of the
space of holomorphic sections $H^0(M, L^k)$ associated to the
$k$th power of $L \to M$ is defined by the  monomials $z^{\alpha}$
where $\alpha$ is a lattice point in the $k$th dilate of the
polytope, $\alpha \in k P \cap \N^m.$ That is, there exists an
invariant frame $e$ over the open orbit so that $s_{\alpha}(z) =
z^{\alpha} e^k$. We denote the dimension of $H^0(M, L^k)$ by $N_k$.
We equip $L$ with a toric Hermitian metric $h = h_0$ whose
curvature $(1,1)$ form $\omega_0 = i \ddbar \log ||e||_{h_0}^2$
lies in $\hcal.$ We often express the norm in terms of a local
\kahler potential, $||e||_{h_0}^2 = e^{- \phi}$, so that
$|s_{\alpha}(z)|_{h_0^k}^2 = |z^{\alpha}|^2 e^{- k \phi (z)}$ for
$s_{\alpha} \in H^0(M, L^k)$.

The monomials are orthogonal with respect to any
such toric inner product and have the norm-squares
\begin{equation} \label{QFORM} Q_{h^k}(\alpha) = ||s_{\alpha}||_{h^k}^2 = \int_{\C^m} |z^{\alpha}|^2 e^{-
k \phi(z)} dV_{\phi}(z), \end{equation} where $dV_{\phi} = (i
\ddbar \phi)^m/ m!$.


\subsection{Coordinates near the divisor at infinity}

To determine the asymptotics of $P_{h^k}(z,w)$ when at least of $z,w$ lies on the divisor at infinity, we need expressions for the monomials on the divisor at infinity, i.e. we need to define holomorphic 
coordinates valid in neighborhoods of points of $\dcal$. We follow   \cite{SoZ,STZ}; the
 following Lemma is proved in  \cite{STZ}.

\begin{lem}
\label{divisor}
Let $\alpha \in P \cap \Z^{m}$, and $z \in M_{P}$. Then, $\chi_{\alpha}(z)=0$
if and only if
\begin{equation}
\label{Nulset}
\mu (z) \in \bigcup \{\bar{F}\,;\,F \mbox{ is a facet }\alpha \not \in \bar{F}\}.
\end{equation}
\end{lem}

For each vertex $v_0 \in P$,
we define  the chart $U_{v_0}$
 by
\begin{equation}
U_{v_{0}}:=\{z \in M_{P}\,;\,\chi_{v_{0}}(z) \neq 0\},
\end{equation}  Since $P$ is Delzant, we can choose lattice points
$\alpha^{1},\ldots,\alpha^{m}$ in $P$ such that each $\alpha^{j}$
is in an edge incident to the vertex $v_{0}$, and the vectors
$v^{j}:=\alpha^{j}-v_{0}$ form a basis of $\Z^{m}$. We define
\begin{equation}
\label{CChange} \eta:(\C^{*})^{m} \to (\C^{*})^{m}, \quad
\eta(z)=\eta_{j}(z):=(z^{v^{1}},\ldots,z^{v^{m}}).
\end{equation}
The map $\eta$ is a $\T$-equivariant biholomorphism with  inverse
\begin{equation}  z:(\C^{*})^{m} \to (\C^{*})^{m},\quad z(\eta)=(\eta^{\Gamma
e^{1}},\ldots,\eta^{\Gamma e^{m}}),
\end{equation}
where $e^{j}$ is the standard basis for $\C^{m}$, and $\Gamma$ is
an $m \times m$-matrix with $\det \Gamma =\pm 1$ and integer
coefficients defined by
\begin{equation}\label{GAMMADEF}
\Gamma v^{j}=e^{j},\quad v^{j}=\alpha^{j}-v_{0}.
\end{equation}Then
\[
\chi_{\alpha^{j}}(z)= (\prod \eta_{j}(z))\chi_{v^{0}}(z),\quad
z \in (\C^{*})^{m}, \;\;\; (\prod \eta_{j}(z)) := \prod_{j =1}^m z^{v^{j}}).
\]

 The
corner of $P$ at $v_0$ is transformed to the standard corner of
the orthant $\R_+^m$ by
 the affine linear transformation
\begin{equation}\label{GAMMATWDEF}
\tilde{\Gamma}:\R^{m} \ni u \to \Gamma u -\Gamma v_{0} \in \R^{m},
\end{equation}
which preserves $\Z^{m}$, carries $P$ to a polytope $Q_{v_0}
\subset \{x \in \R^{m}\,;\,x_{j} \geq 0\}$ and carries the facets
$F_j$ incident at $v_0$ to the coordinate hyperplanes $=\{x \in
Q_{v_0}\,;\,x_{j}=0\}$. The  map $\eta$ extends a homeomorphism:
\begin{equation}
\eta:U_{v_{0}} \to \C^{m},\quad \eta(z_{0})=0,\quad z_{0}=\mbox{
the fixed point corresponding to } v_{0}.
\end{equation}
By this homeomorphism, the set $\mu_{P}^{-1}(\bar{F}_{j})$
corresponds to the set $\{\eta \in \C^{m}\,;\,\eta_{j}=0\}$. If
$\bar{F}$ be a closed face with $\dim F=m-r$ which contains
$v_{0}$, then there are facets  $F_{i_{1}},\ldots,F_{i_{r}}$
incident at $v_0$ such that $\bar{F}=\bar{F}_{i_{1}} \cap \cdots
\cap \bar{F}_{i_{r}}$. The subvariety $\mu_{P}^{-1}(\bar{F})$
corresponding $\bar{F}$ is expressed by
\begin{equation}\label{ETAIJ}
\mu_{P}^{-1}(\bar{F}) \cap U_{v_{0}}= \{\eta \in
\C^{m}\,;\,\eta_{i_{j}}=0,\quad j=1,\ldots,r\}.
\end{equation}
When working near a point of $\mu_{P}^{-1}(\bar{F})$, we simplify
notation by writing
  \begin{equation} \label{PRIMECOORD} \eta=(\eta',\eta'') \in
\C^{m}=\C^{r} \times \C^{m-r} \end{equation}  where $\eta' =
(\eta_{i_j})$ as in (\ref{ETAIJ}) and where $\eta''$ are the
remaining $\eta_j$'s, so that  $(0,\eta'')$ is a local coordinate
of the submanifold $\mu_{P}^{-1}(\bar{F})$. When the point $(0,
\eta'') $ lies in the open orbit of $\mu_{P}^{-1}(\bar{F})$, we
often write $\eta'' = e^{i \theta'' + \rho''/2}.$

These coordinates may be described more geometrically as {\it
slice-orbit} coordinates \cite{SoZ}. Let $P_0 \in \mu_{P}^{-1}(\bar{F})$ and
let $(\C^*)^m_{P_0}$ denote its stabilizer (isotropy) subgroup.
Then there always exists a local slice at $P_0$, i.e., a local
analytic subspace $S \subset M$ such that  $P_0 \in S$, $S$ is
invariant under $(\C^*)^m_{P_0}$,  and such that  the natural
$(\C^*)^m$ equivariant map of the normal bundle of the orbit
$(\C^*)^m \cdot P_0$,
\begin{equation} \label{SLICE} [\zeta, P] \in (\C^*)^m
\times_{(\C^*)^m_z} S \to \zeta \cdot P \in M
\end{equation}  is a biholomorphism onto $(\C^*)^m \cdot S$. The  slice $S$ can be
taken to be the image of a ball in  the hermitian normal space
$T_{P_0} ((\C^*)^m P_0)^{\perp}$ to the orbit under any local
holomorphic embedding  $w: T_{P_0} ((\C^*)^m P_0)^{\perp} \to M$
with $w(P_0) = P_0, dw_{P_0} = Id.$  The affine coordinates
$\eta''$ above define the slice $S = \eta^{-1} \{(z', z''(P_0)):
z' \in (\C^*)^{r}\}$.  The local `orbit-slice' coordinates are
then defined by
\begin{equation} \label{OS} P = (z', e^{i \theta'' + \rho''/2})
\iff \eta(P) = e^{i \theta'' + \rho''/2} (z', 0)
\end{equation} where $(z', 0) \in S$ is  the point on the slice  with affine holomorphic
coordinates $z' = (\eta')$.

For  simplicity of notation we  suppress the transformation
$\tilde{\Gamma}$ and coordinates $\eta$, and  we will use the
`orbit-slice' coordinates of (\ref{OS}). Thus, we  denote the
monomials cooresponding to lattice points $\alpha$ near a face $F$
by \begin{equation} \label{CHIALPHA} \chi_{\alpha', \alpha''}(z', e^{ \langle (i \theta'' + \rho''/2)}): =  (z')^{\alpha'} e^{ \langle (i \theta'' + \rho''/2), \alpha'' 
\rangle}, \end{equation} where $\tilde{\Gamma}(\alpha) = (\alpha', \alpha'')$
with  $\alpha''$
 in the coordinate hyperplane corresponding under
$\tilde{\Gamma}$ to $F$ and with $\alpha'$ in the normal space.

 \subsection{\label{TCD} Toric Calabi diastasis}

 In the toric case, the Kaehler potential $\phi(z,w)$ is $\T$-invariant, and so
 \begin{equation} \label{phitilde} \phi(e^{\rho/2 + i \theta)}) = \tilde{\phi} (\rho) \end{equation}
 where $\tilde{\phi} (\rho)$ is a smooth convex function on $\R^m$. We also
 write $\phi(z) = F(|z|^2)$ where $F(e^{\rho}) = \tilde{\phi}(\rho).$ Then
 \begin{equation} \label{AAEF} F_{\C} (z \cdot \bar{w}) = \; \mbox{the almost analytic extension of }\;
     F(|z|^2)\; \mbox{to}\; M \times M.  \end{equation}
    Hence,
    \begin{equation} \label{TORDIA} -D(z,w) = F(z \cdot \bar{w}) + F(w \cdot \bar{z}) - F(|z|^2) - F(|w|^2). \end{equation}
    
As an example,  the  Bargmann-Fock(-Heisenberg) 
 Kaehler potential is $|z|^2$ and
the  analytic extension is
  $F_{BF, \C}(z,w) = z \cdot \bar{w} $ and $D(z,w) = |z|^2 + |w|^2 - 2 \Re z \bar{w} = |z -w|^2. $
  The example of the Fubini-Study metric on $\CP^m$ was discussed in the introduction.

We now give formulae for the diastasis when two points lie on the same $\R_+^m$ orbit.
    If we write $z = e^{\rho_1/2 + i \theta_1}, w = e^{\rho_2/2 - i \theta_2}$,  and  $\tilde{F}(\rho) = F(e^{\rho})$, then

   
       \begin{lem}\label{CDMSAME}
       \begin{itemize}
       
       \item
       When $z, w\in M^o$ lie on the same $\R_+^m$ orbit, i.e. if $\theta_1 = \theta_2$ then
      \begin{equation} \label{TORDIAc} - \half D(z,w) = \tilde{F}(\half(\rho_1 + \rho_2)) -  \half(\tilde{F}(\rho_1) +  \tilde{F}(\rho_2)). \end{equation}
             
             \item 
             If $z,w \in F$ for some face $F$ of $P$, then in the coordinates \eqref{OS} with $z = (0, e^{i \theta_1'' + \rho_1'/2}),
             w =  (0, e^{i \theta_2'' + \rho_2''/2}), $
                  
                            \begin{equation} \label{TORDIAc2} - \half D(z,w) = \tilde{F}(\half(\rho_1'' + \rho_2'')) -  \half(\tilde{F}(\rho_1'') +  \tilde{F}(\rho_2'')).
                \end{equation}
                
               
               \end{itemize}
               
             \end{lem}
             Since $F$ is convex, the right sides are negative, i.e. $D(z,w) > 0$. The formula for $z,w \in F$ reflects the fact that
             the diastasis of a complex submanifold is the restriction of the diastasis on the ambient manifold \cite{Ca53}.

\subsection{\label{TBK} Toric Bergman kernels }

We now justify \eqref{Bhkintro}.

\begin{lem} \label{BNTORIC}The toric Bergman kernel is given on the open orbit by 
\begin{equation}B_{h^k}(z, w)=
\sum_{ \alpha \in k P}  \frac{z^{\alpha}
\bar{w}^{\alpha}}{Q_{h^k}(\alpha)} \;.\end{equation} 
\end{lem}

\begin{proof} We recall that
$\chi_{\alpha}(z) = z^{\alpha}$ is the local representative of
$s_{\alpha}$ in the open orbit with respect to an invariant frame.
Since $\{\frac{\chi_{\alpha}}{\sqrt{\QQ_{h^k}(\alpha)}}\}$ is the
local expression of an  orthonormal basis, we have
$$B_{h^k}(z,w) = \sum_{\alpha \in k P \cap \Z^m} \frac{\chi_{\alpha}(z)
\overline{\chi_{\alpha}(w)} }{\QQ_{h^k}(\alpha)} $$ hence
\begin{equation} \label{HATPI} \hat{\Pi}_{h^k}(z,0; w, 0) = \sum_{\alpha \in k P \cap \Z^m} \frac{\chi_{\alpha}(z)
\overline{\chi_{\alpha}(w)}e^{- k (\phi(z) + \phi(w))/2}
}{\QQ_{h^k}(\alpha)}.
\end{equation}

In the
case of a toric variety with $0 \in \bar{P}$, there exists a frame
$e$ such that $s_{\alpha}(z) = z^{\alpha} e$ on the open orbit,
and then the Bergman kernel takes the form \eqref{Bhkintro}.
\end{proof}

\section{\label{PROOFSECT} Off-diagonal decay: Proof of Theorem  \ref{<}}

We break up the discussion into three cases accordingly as $z, w \in M^o$,
or one (or both) lie on $\dcal$. The discussion is similar to that  of the diastasis in \S \ref{TCD}.

\subsection{Case (i): $z, w \in M^o$}

We now estimate $ |B_{h^k}(z, w)|$ on a toric \kahler manifold when $z,w \in M^o$.
We write $z = e^{\rho/2 + i \theta}$ where $\rho \in \R^m$ and $e^{i \theta} \in \T$
and (as in \eqref{phitilde}) $\tilde{\phi}(\rho) = \phi(e^{\rho/2})$.
Then $z^{\alpha} = e^{\langle \alpha, \rho_1 \rangle/2 + i \langle \theta_1, \alpha\rangle}$ and  $w^{\alpha} = e^{\langle \alpha, \rho_2 \rangle/2 + i \langle \theta_2, \alpha\rangle}$. Hence
\begin{equation}\label{BNbb} |B_{h^k}(z, w)|= \left|
\sum_{\alpha \in k P} \frac{e^{ \langle \alpha, \half (\rho_1 + \rho_2)/2 \rangle
+ i \langle \alpha, \theta_1 -\theta_2 \rangle}}{Q_{h^k}(\alpha)} \right| \;.\end{equation}Thus,
\begin{equation}
 P_{h^k}(z,w) =  \left|
\sum_{\alpha \in k P} \frac{e^{\langle \alpha, \half (\rho_1 + \rho_2)/2 \rangle
+ i \langle \alpha, \theta_1 -\theta_2 \rangle}}{Q_{h^k}(\alpha)} \right| e^{-\half k( \tilde{\phi}(\rho_1) +  \tilde{\phi}(\rho_2))}. \end{equation}

We  first consider a special case:

\subsection{$z$ and $w$ have the same $\theta$ coordinate}

Suppose that   $z = e^{\rho_1/2 + i \theta}$ and $w = e^{\rho_2/2 + i \theta}$. 
\begin{equation}\label{BNc} |B_{h^k}(z, w)|= 
\sum_{\alpha \in k P}  \frac{e^{\langle \alpha, \half(\rho_1 + \rho_2) \rangle}}{Q_{h^k}(\alpha)}  \;.\end{equation}
This is equivalent to an on-diagonal value of the Bergman kernel: 
$$B_{h^k}(e^{\rho'/2}, e^{\rho'/2}) = \sum_{\alpha \in k P}  \frac{e^{\langle \alpha, \rho' \rangle}}{Q_{h^k}(\alpha)},$$
with $\rho' = \half ( \rho_1 + \rho_2).$ By the diagonal Bergman kernel asymptotics
\eqref{TYZ},
\begin{equation}\label{BNd} |P_{h^k}( e^{\rho_1/2 + i \theta}, e^{\rho_2 /2+ i \theta})|
\simeq  k^m e^{k \tilde{\phi}(\half(\rho_1 + \rho_2))} e^{- k \half(\tilde{\phi}(\rho_1) +\tilde{ \phi}(\rho_2))} \;.\end{equation}
Since $\tilde{\phi}$ is convex,
$$\tilde{\phi}(\half(\rho_1 + \rho_2)) < \half (\tilde{\phi}(\rho_1) + \tilde{\phi}(\rho_2)).$$
It follows that 
$$\lim_{k \to \infty}\frac{1}{k} \log |P_{h^k}( e^{\rho_1/2 + i \theta}, e^{\rho_2 + i \theta})|
= \tilde{\phi}(\half(\rho_1 + \rho_2)) -  \half(\tilde{\phi}(\rho_1) + (\tilde{\phi}(\rho_2)) < 0. $$

\begin{rem} As mentioned in the introduction, the assumption that $z,w$ have the same $\theta$-coordinate
could be put more intrinsically by saying that $z,w$ lie on the same orbit of the $\R_+$-action defined
by the $(\C^*)^m$ action.

\end{rem}

\subsection{$z,w \in M^o$  but do not lie on the same $\R_+^k$ orbit}

It is obvious that 

\begin{equation}\label{BNbc}\begin{array}{lll}  |B_{h^k}(z, w)| & = & \left|
\sum_{\alpha \in k P} \frac{e^{ \langle \alpha, \half (\rho_1 + \rho_2) \rangle
+ i \langle \alpha, \theta_1 -\theta_2 \rangle}}{Q_{h^k}(\alpha)} \right|\\&&\\
& \leq &  
\sum_{\alpha \in k P} \frac{e^{i \langle \alpha, \half (\rho_1 + \rho_2) \rangle
}}{Q_{h^k}(\alpha)} \\ \end{array} \;.\end{equation}
Arguing as in the previous case, 
$$\limsup_{k \to \infty}\frac{1}{k} \log |P_{h^k}( e^{\rho_1/2 + i \theta}, e^{\rho_2 + i \theta_2})|
\leq  \tilde{\phi}(\half(\rho_1 + \rho_2)) -  \half(\tilde{\phi}(\rho_1) + (\tilde{\phi}(\rho_2)) < 0. $$

\subsection{Case (ii)\;\;$z \in M^o, w \in \dcal$}

In this, case, we must use the slice-orbit coordinates adapted to the face $F$ containing $w$. Of course,
$z,w$ lie on different $\R^m$ orbits (even $(\C^*)^m$-orbits), and we can only give an upper bound then. 
We use the coordinates in  \eqref{OS}, so that $z =  (z', e^{i \theta_1'' + \rho_1''/2})$ and $w =(0,   e^{i \theta_2'' + \rho_1''/2}))$
and further write $z' = e^{\rho'/2 + i \theta'}.$ Then,
\begin{equation}\label{BNbc} B_{h^k}(z, w)= 
\sum_{\alpha \in k P} \frac{e^{i \langle \alpha', \rho' + i\theta' \rangle}  e^{ \langle \alpha'', (\rho_1''/2 + \rho_2''/2\rangle
+ i \langle \alpha'', (\theta_1''-  \theta_2'')\rangle}}{Q_{h^k}(\alpha)}. \end{equation}
If $\theta' = 0, \theta_1'' =\theta_2''$ then 
\begin{equation}\label{BNbd} |B_{h^k}(z, w)|= 
\sum_{\alpha \in k P} \frac{e^{i \langle \alpha', \rho'\rangle}  e^{ \langle \alpha'', (\rho_1''/2 + \rho_2''/2\rangle}}{Q_{h^k}(\alpha)}. \end{equation}
Otherwise this is an upper bound. This is the value on the diagonal of the Bergman kernel at the point with coordinates
$$(e^{\half \rho'}, e^{\frac{1}{4} (\rho_1'' + \rho_2'') }). $$
We again use the diagonal Bergman kernel asymptotics \eqref{TYZ}. We cannot use the open-orbit \kahler potential since it
is not well-defined on the divisor $\dcal$. But the diastasis is independent of the choice of \kahler potential and by the
previous calculation we find that $\frac{1}{k} \log P_{h^k} (z,w)$ tends to $D(z,w)$.

\section{\label{CHRISTSECT} Off-diagonal Bergman kernel for M. Christ's metrics}

In this section, we briefly explain the modification of Theorem \ref{MAIN} which
applies to M. Christ's metrics \cite{Chr13}.  He considers
 Kaehler metrics on $\C^m$ whose potentials $\phi(z) = \phi(x + i y)$
are functions of $x$ alone. Such metrics are analogous to  toric metrics with the torus
replaced by $\R^m$. More precisely, we  think of $\R^m_x $ and $\R^m_y$ as two additive subgroups
of $\C^m$, both with non-compact orbits. The difference between toric Kaehler metrics
and those of \cite{Chr13} is that the $e^{i \theta}$ action with compact torus
orbits is replaced by the additive $i \R^n$ action with non-compact orbits,
and the metric is invariant under this action. Since these metrics are similar
to toric metrics except that the torus is `unravelled' to $\R^m$, we refer
to them as `non-compact toric' metrics.  To our knowledge, there does not
exist a standard term in complex geometry.

For non-compact toric Kaehler manifolds, one has an $\R^m$ action replacing
the $\T$ actions and therefore a continuous spectrum of joint eigenfunctions.
The image of the moment map of the $\R^m$ action is all of $\R^m$ rather
than a convex Delzant polytope $P$. The exponents $\alpha$ are lattice points
in $kP$ in the toric case, but are simply elements $\alpha \in \R^m$ for every
power of the line bundle in the non-compact case. In the toric case one
writes the monomial $e^{\langle \alpha, \rho/2 + i \theta}$ in angle-action
coordinates and in the non-compact case one writes the joint eigenfunctions as
$e^{\langle \alpha, \rho/2 + i y}$ with $\alpha, y \in \R^m$.

The weighted Hilbert space in the non-compact case for the $\lambda$-th
power of the positive Hermitian line bundle $L \to \C^m$  is the set of entire
holomorphic functions $f$ so that
\begin{equation} ||f||^2_{L^2(X, L^{\lambda}} = \int_{\C^m} |f(z)|^2 e^{- \lambda \phi(z)} dm(z) < \infty. \end{equation}
It is assume that the curvature form of $\phi$ is strictly positive and uniformly 
bounded above and below. Thus
the real Hessian of $\phi$ is a positive matrix $\rm{Hess}(\phi)(x)$ satisfying
$$C^{-1} |v|^2 \leq \langle \rm{Hess}(\phi)(x) v,v \rangle \leq C |v|^2.$$
Note that such a Kaehler metric is quite different from the usual Bargmann-Fock
type Kaehler metrics such as $|z|^2$, since the latter are invariant under the compact torus action (i.e. are non-compact toric Kaehler manifolds in the standard sense).

Theorem 2.1 of \cite{Chr13} says that if  there exists an open set $U \subset \C^m$ so that, for any $\delta > 0$, there exists a sequence $\lambda_{\nu}
\to \infty$ such that $|B(z,z') \leq e^{- \epsilon \lambda_{\nu}}$ for some $\epsilon > 0$
and for  $(z,z') \in U$, $|z - z'| \geq \delta$, then the Kaehler potential $\phi$ is real analytic on $U$.

 Of course, the joint eigenfunctions $e^{\langle \alpha, x/2+ i y\rangle}$ do not
 lie in $L^2$ since the integrals
 $$ \int_{\R^{2m}}  e^{- \langle \alpha, x \rangle} e^{- \lambda \phi(x)} dx dy  $$
 diverges due to the lack of damping in the $y$ variable.  We may however
 express the Bergman kernel in terms of a generalized orthonormal basis
 of joint eigenfunctions since
  $$ \int_{\R^{2m}}  e^{ \langle (\alpha + \beta), x/2 } e^{i\langle \alpha - \beta, y \rangle} e^{- \lambda \phi(x)} dx dy  = \delta(\alpha - \beta)  \left(\int_{\R^{m}}  e^{ \langle \alpha, x  \rangle} e^{- \lambda \phi(x)} dx\right).$$ 
  
  Analogously to \eqref{QFORM}, we write
  \begin{equation} Q_{\lambda \phi}(\alpha) = \int_{\R^{m}}  e^{ \langle \alpha, x  \rangle} e^{- \lambda \phi(x)} dx.\end{equation}

  \begin{prop} \label{CHRBK} The Bergman kernel $B_{\lambda}(z,w)$ of Christ's metric 
  has the form,
  $$B_{\lambda}(z,w) = \int_{\R^m} e^{ \langle \alpha, (x + x')/2 + i (y - y') \rangle}
  \frac{ d \alpha}{Q_{\lambda \phi}(\alpha) }. $$
  \end{prop}
  In the present  terminology, the Bergman kernel is the kernel which
  is holomorphic in $z$, anti-holomorphic in $w$ and which represents
  the \szego kernel with respect to a local frame $e_L^{\lambda}$. The setting
  of the trivial line bundle $L  = \C^m \times \C \to \C^m$ means that the constant
  function $1$ can be used as a frame, but its norm is $e^{- \lambda \phi(z)/2}$. 
  Thus the \szego kernel is
  \begin{equation} \label{Szego} \Pi_{\lambda}(z,w) = 
  B_{\lambda}(z,w)e^{- \lambda \phi(z)/2} e^{- \lambda \phi(w)/2}. \end{equation}
  
  Now let us suppose that $y = y'$ but $x \not= x'$. Then
   $$B_{\lambda}(z,w) = \int_{\R^m} e^{ \langle \alpha, (x + x')/2 \rangle}
  \frac{ d \alpha}{Q_{\lambda \phi}(\alpha) }. $$
  As in the toric case,  we recognize that this is closely related to the diagonal value of the Bergman kernel
  at the point $(\frac{x - x'}{2}, y)$ for any $y$, i.e.
  $$B_{\lambda}((\frac{x + x'}{2}, y),(\frac{x + x'}{2}, y)) = \int_{\R^m} e^{\lambda \langle \alpha, (x + x')/2 \rangle}
  \frac{ d \alpha}{Q_{\lambda \phi}(\alpha) }. $$
  However as above, the off-diagonal \szego kernel does not equal the
  corresponding on-diagonal \szego kernel at the mid-point but differs by the
  metric factor. That is,
  \begin{equation}\begin{array}{lll} \Pi_{\lambda}((x/2 + i y, x'/2 + i y') & = & B_{\lambda}(x/2 + iy,
  x'/2 + i y') e^{- \lambda \phi(z)/2} e^{-\lambda \phi(z')/2} \\ &&\\
  & = & 
   B_{\lambda}((\frac{x + x'}{2}, y),(\frac{x + x'}{2}, y)) e^{- \lambda \phi(\frac{x + x'}{2}, y) }\\&&\\&& \times \left(
   e^{- \lambda \phi(z)/2} e^{-\lambda \phi(z')/2}  e^{ \lambda \phi(\frac{x + x'}{2}, y)} \right) \\&&\\
   & = & 
   \Pi_{\lambda}((\frac{x + x'}{2}, y),(\frac{x + x'}{2}, y))   \left(
   e^{- \lambda \phi(z)/2} e^{-\lambda \phi(z')/2}  e^{ \lambda \phi(\frac{x + x'}{2}, y)} \right)  \\&&\\
   & = & 
   \Pi_{\lambda}((\frac{x + x'}{2}, y),(\frac{x + x'}{2}, y))   \left(
   e^{- \lambda \phi(x)/2} e^{-\lambda \phi(x')/2}  e^{ \lambda \phi(\frac{x + x'}{2})} \right) . \end{array} \end{equation}

  By the Bergman kernel asymptotics, 
   $$\Pi_{\lambda}((\frac{x + x'}{2}, y),(\frac{x + x'}{2}, y)) \sim \lambda^m + O(\lambda^{m-1}). $$
   On the other hand, by convexity of $\phi(x)$,$$ e^{- \lambda \phi(x)/2} e^{-\lambda \phi(x')/2}  e^{ \lambda \phi(\frac{x + x'}{2})}$$
   is exponentially decaying at speed $\lambda$, since
   $$\phi(x)/2 +  \phi(x')/2  \geq \phi(\frac{x + x'}{2}). $$
   
  This argument does not give exponential decay when $x = x'$ and therefore
  is consistent with the results of \cite{Chr13}.


\begin{thebibliography}{HHHH}

\bibitem[Ca53]{Ca53} E. Calabi, Isometric imbedding of complex manifolds. Ann. of Math. (2) 58, (1953). 1-23. 

\bibitem[Chr13]{Chr13}  M. Christ, 
Off-diagonal decay of Bergman kernels: On a conjecture of Zelditch, arXiv:1308.5644.

\bibitem[Chr13B]{Chr13B} M. Christ,
Upper bounds for Bergman kernels associated to positive line bundles with smooth Hermitian metrics,   arXiv:1308.0062.


\bibitem[Chr03]{Chr03} M. Christ,
Slow off-diagonal decay for Szegö kernels associated to smooth Hermitian line bundles. {\it Harmonic analysis at Mount Holyoke} (South Hadley, MA, 2001), 77–89, 
Contemp. Math., 320, Amer. Math. Soc., Providence, RI, 2003.

\bibitem[PhSt]{PhSt} D.H. Phong and J. Sturm, The Monge-Ampere operator and geodesics in the space of Kaehler potentials, Invent. Math. 166 (2006), 125-149.

\bibitem[Sj13]{Sj13} J. Sjostrand, Notes on Bergman kernels (private communication).

\bibitem[RZ]{RZ} Y. A. Rubinstein and S.  Zelditch, The Cauchy problem for the homogeneous Monge-Ampère equation, I. Toeplitz quantization. J. Differential Geom. 90 (2012), no. 2, 303-327.


\bibitem[STZ]{STZ} B. Shiffman, T. Tate, and S.  Zelditch,  Distribution laws for integrable eigenfunctions. Ann. Inst. Fourier (Grenoble) 54 (2004), no. 5, 1497-1546,

\bibitem[SoZ]{SoZ} J. Song and S. Zelditch,  Bergman metrics and geodesics in the space of Kähler metrics on toric varieties. Anal. PDE 3 (2010), no. 3, 295-358.

\bibitem[SoZ2]{SoZ2} J. Song and S.  Zelditch, Test configurations, large deviations and geodesic rays on toric varieties. Adv. Math. 229 (2012), no. 4, 2338-2378.

\bibitem[Ze1]{Z} S. Zelditch, {\it \szego kernels and a theorem of Tian},
Int. Math. Res. Notices, no.6 (1998), 317--331.

\end{thebibliography}
\end{document}